\documentclass[11pt]{amsart}
\usepackage{amsmath,amssymb,amsthm, verbatim, amscd, amsfonts, yfonts, amscd}
\usepackage{setspace}

\newcommand{\V}{\ensuremath{\vspace{.1in}}}

\newcommand{\R}{\ensuremath{\mathbb{R}}}
\newcommand{\C}{\ensuremath{\mathbb{C}}}
\newcommand{\N}{\ensuremath{\mathbb{N}}}
\newcommand{\Z}{\ensuremath{\mathbb{Z}}}
\newcommand{\Q}{\ensuremath{\mathbb{Q}}}

\theoremstyle{definition}

\theoremstyle{remark}

\theoremstyle{plain}

\newtheorem{theorem}{Theorem}[section]

\newtheorem{lemma}{Lemma}[section]

\newtheorem{prop}[theorem]{Proposition}

\newtheorem{kor}{Corollary}[theorem]

\newtheorem{corlem}{Corollary}[lemma]
\newtheorem{propcor}{Corollary}[theorem]

\theoremstyle{definition}

\title{ A Generalization of Wantzel's Theorem,\\ m-sectable angles,\\and\\the  density of certain Chebyshev-polynomial images}
\author{Peter J. Kahn}

\begin{document}

\maketitle
\markboth{\large{Peter J. Kahn}}{\large {A Generalization of Wantzel's Theorem, etc.}}

\vspace{.2in}

\begin{center}
Department of Mathematics\\
Cornell University\\
Ithaca, New York 14853\\
January 2, 2014\\
\end{center}

\vspace{.2in}

\begin{abstract}
The eponymous theorem of P.L. Wantzel \cite{wan}  presents a necessary and sufficient  criterion for angle trisectability in terms of the third Chebyshev polynomial $T_3$, thus making it easy to prove that there exist non-trisectable angles. We generalize this theorem to the case of all Chebyshev polynomials $T_m$ (Corollary \ref{C:generalized wantzel}).  We also study the set \textbf{m-Sect} consisting of all cosines of $m$-sectable angles (see \S 1), showing  that, when $m$ is not a power of two, \textbf{m-Sect} contains only algebraic numbers (Theorem \ref{T:algebraicity}). We then introduce a notion of density based on the diophantine-geometric concept of height of an algebraic number and obtain a result on the density of certain polynomial images.  Using this in conjunction with the Generalized Wantzel Theorem, we obtain our main result: for every real algebraic number field $K$,  the set  \textbf{m-Sect}\ $\cap\ K$  has  density zero in  $[-1,1]\ \cap\  K$  when $m$ is not a power of two (Corollary \ref{C:m-sect estimate}).
\end{abstract}

\section{Introduction}

This paper poses and answers some interesting algebraic questions raised by P.L. Wantzel's 1837 theorem that destroyed the age-old hope of finding a ``ruler and compass'' construction for angle-trisection.\footnote{See \cite{k}, pp. 2, 3, for a thumbnail sketch of the history of this problem leading to Wantzel's work.}  More precisely, Wantzel \cite{wan} proved the following result:


\noindent\textbf{Wantzel's Theorem:}\quad \emph{Let $\alpha$ be any angle, and set $\cos(\alpha)=a$.\footnote{We continue with this notation throughout this introduction.}  Then $\alpha$ admits a trisection using only an unmarked straightedge and compass if and only if the polynomial $4x^3-3x-a$ has a zero in the field $\Q(a)$.}


It is easy to see, as Wantzel did, that when $4x^3-3x-a$ satisfies the algebraic criterion of his theorem, the number $a$ must be algebraic.\footnote{The converse, however, is false.  For example, there are infinitely many non-trisectable angles whose cosines are rational numbers.  E.g, see Lemma 2.4 (b) below, or see \cite{k}, p. 8 ff.}  Thus, many (in fact, most) angles are not trisectable.


Here are four questions suggested by Wantzel's Theorem.


The first question involves extending or generalizing the theorem.  Let $m$ be any positive integer.  We say that $\alpha$ is $m$-sectable if it admits an $m$-fold equipartition by a construction that uses only an unmarked straightedge and compass.  When $m=2$ (resp., $m=3$) we use the familiar terms bisectable (resp., trisectable) instead. For a given $m$,  we say that \emph{$m$-sectability always holds} if every angle is $m$-sectable. Otherwise we say that \emph{$m$-sectability sometimes fails}. Wantzel's Theorem shows that trisectability sometimes fails.  We can now ask the following:


\textbf{(A)}\quad \emph{Can we extend Wantzel's Theorem to the case of $m$-sectability, for $m>3$?}


This question requires some preliminary discussion, which we defer.  Instead we ask an easier question:


\textbf{(B)}\quad  \emph{Suppose $\alpha$ is $m$-sectable for some $m$.  Must $a=\cos(\alpha)$ be an algebraic number?}


This has a fairly easy, direct answer:


\begin{theorem}~\label{T:algebraicity}  If $m$ is a power of two, then $m$-sectability always holds.  In other words,  the quantity $a$ can assume any value in the unit interval $[-1,1]$. However,  when $m$ is not a power of two, $\alpha$ is $m$-sectable only if $a$ is an algebraic number in $[-1,1]$.
\end{theorem}


The first sentence of the theorem is obviously true, since bisectability always holds. It is mentioned only for completeness.


We denote the field of algebraic numbers by $\overline{\Q}$, and we let \textbf{m-Sect} denote the set of cosines of $m$-sectable angles. By Theorem \ref{T:algebraicity}, when $m$ is not a power of two, we have an inclusion of countable sets
 \[ \mbox{\textbf{m-Sect}} \subseteq\overline{\Q}\cap [-1,1].\]
We now ask:


\textbf{(C)}\quad \emph{When $m$ is not a power of two, how densely is \textbf{m-Sect} distributed in  $\overline{\Q}\cap [-1,1]$?}


The notion of density that we use is tied to the concept of \emph{height}
of an algebraic number; we describe this briefly in \S 3. A comprehensive discussion of height may be found in \cite{lan}.  Here, we say only that, given an algebraic number field $K$,  there is a  function $H_K:K\rightarrow [1,\infty)$, called the \emph{height function on $K$}, with the important property that, for every real number $B\in [1,\infty)$, the set $H_K^{-1}[1,B]$ is non-empty and finite.

We consider sets $S\subseteq T\subseteq\C$ such that $1\in T$. We then define the \emph{$K$-density of $S$ in $T$} to be the limit as $B\rightarrow \infty$ of the quotients of finite cardinalities

\begin{equation}
\delta_K(S,T;B)= \frac{|S\cap H_K^{-1}[1,B]| }{|T\cap H_K^{-1}[1,B]|},
\end{equation}
provided this limit exists.  We denote the limit by $\delta_K(S, T)$.


\begin{theorem}~\label{T:density of m-sect}  Let $K$ be an algebraic number field $\subset \R$, and let $m$ be any positive integer.  Then, the $K$-density $\delta_K(\mathbf{m-Sect}, [-1,1])$ exists. It equals $1$ when $m$ is a power of two,  and it equals $0$ otherwise.
\end{theorem}


 The first assertion in the last sentence is immediate from Theorem \ref{T:algebraicity}. The second  is a consequence of Corollary \ref{C:m-sect estimate} below. An analogue of this,  in which  the infinite-dimensional field $\overline{\Q}$ replaces $K$,  is still out of reach. If true, it seems to require numerical estimates more delicate than those used here (in \S 3). See the remark after Proposition \ref{P:polynomial image estimate} and Corollary \ref{C:m-sect estimate}.


We now return to question \textbf{(A)}.  The polynomial $4x^3-3x$ appearing in Wantzel's Theorem will no doubt be recognized by many as the \emph{third Chebyshev polynomial}
$T_3(x)$.  Its connection with angle trisectability is clearly a consequence of the well-known trigonometric identity $\cos(3\beta)= T_3(\cos(\beta))$; an analogue holds for every positive integer $m$ and every angle $\beta$ (cf., Lemma \ref{L:trig formula}).  Thus, it is natural to try to generalize Wantzel's Theorem to the case of $m$-sectability by replacing $T_3(x)$ by the $m^{th}$ Chebyshev polynomial $T_m(x)$.
Indeed, if we consult the theory of constructible  numbers, with which we assume the reader has some  familiarity, together with the basic definition and properties of $T_m(x)$ (cf., \S 2), we  see that the following is an immediate consequence of definitions:


\begin{prop}~\label{P:constructibility criterion} Recall that, for an angle $\alpha$, we set $a = \cos(\alpha)$.  Then $\alpha$ is $m$-sectable if and only if the  polynomial $T_m(x)- a$ has a zero that is constructible over the field $\Q(a)$.\quad$\square$
\end{prop}


One feature of a real number constructible over a field $F\subset \R$ is that it is algebraic over $F$ with minimal polynomial of degree a power of two.  Thus,  if a real zero of the cubic polynomial $T_3(x)- a$  is constructible over $\Q(a)$,  then $T_3(x)- a$ must have a linear factor over
$\Q(a)$, i.e., a zero \emph{that belongs to} $\Q(a)$. Therefore, Wantzel's Theorem follows from  Proposition \ref{P:constructibility criterion}.  This is not surprising since Wantzel's argument (which predated the development of field theory) amounts to an analysis of the notion of ruler and compass constructibility in algebraic terms, essentially equivalent to the modern formulation.


(We note in passing that the case in which $m$ is a power of two is not an exception to Proposition \ref{P:constructibility criterion}.  For in that case, every angle is $m$-sectable because bisection always holds. And,  moreover, in that case,  $T_m(x)-a$ has a zero constructible over $\Q(a)$ for any $a\in [-1,1]$, as is easy to show inductively using well-known facts about Chebyshev polynomials (cf. \S 2).)


Of significance for us here is that we are interested in answering question \textbf{(C)}, and for that purpose the formulation in Wantzel's Theorem is much more useful than the formulation in Proposition \ref{P:constructibility criterion}.  This is because the assertion ``$T_m(x)-a$ has a zero in $\Q(a)$'' can be rewritten  as an assertion about the image of the polynomial $T_m|\Q(a)$: namely, ``$a\in T_m(\Q(a))$.'' As a consequence,  when the apparently stricter criterion holds, we are able to to answer question \textbf{(C)} by obtaining and applying a somewhat general result about the density of polynomial images (Proposition \ref{P:polynomial image estimate} below).  The criterion that $T_m(x)-a$ have a zero constructible over $\Q(a)$ does not seem to allow such a straightforward application.


In light of this discussion, we are led to ask the following final question:


\textbf{(D)}\quad \emph{ Suppose that $T_m(x)-a$ has a zero constructible over
$\Q(a)$.  Under what conditions on $m$ can we conclude that $T_m(x)-a$ has a zero belonging to $\Q(a)$?}


The following result gives a definitive answer to this question.


\begin{theorem}~\label{T:constructible zeros}\quad  (a)  Suppose that $m$ is even.  Then there exist rational numbers $a$ in $[-1,1]$  such that $T_m(x)-a$ has a zero constructible over $\Q(a)=\Q$  but no zero in \Q.
(b) Suppose $m$ is odd, and let $a$ be any real number in $[-1,1]$.   $T_m(x)-a$ has a zero that is constructible over $\Q(a)$ if and only if it has a zero in $\Q(a)$.
\end{theorem}


We  now use this result to obtain a generalization of Wantzel's Theorem to the case of $m$-sectable angles \underline{for all m} (statement (b) of the following corollary).  This is our answer to question \textbf{(A)}.


\begin{kor}~\label{C:generalized wantzel} (a)\quad When $m$ is even, there exist $m$-sectable angles $\alpha$ such that $a$ is rational but $T_m(x)-a$ has no rational zero. (b)\quad   \hbox{\textbf{\emph{[A Generalized Wantzel Theorem]}}}\quad  Let $m$ be any positive integer, and let $m_{odd}$ be the maximal odd divisor of $m$. Then, $\alpha$ is $m$-sectable if and only if the  polynomial $T_{m_{odd}}(x)-a$ has a zero in $\Q(a)$.
\end{kor}


We now return to question \textbf{(C)}. We answer it by combining the Generalized Wantzel Theorem above with the following result on the density of polynomial images. This in turn is a fairly straightforward consequence of basic facts about heights.  We derive this result in \S 3 (Corollary \ref{C:polynomial image finite density estimate}).


\begin{prop}~\label{P:polynomial image estimate}\quad Let $K$ be an algebraic number field of degree $[K:\Q]=n$.   Choose any  polynomial $f(X)$ in $K[X]$ of degree $d\geq 1$.   Then, there exist positive real numbers $B_0$ and $E_0$, depending only on $f$ and $K$, such that, for $B\geq B_0$,
\begin{equation} \delta_K(f(K),K;B)\leq E_0\cdot B^{(2/d)-2}.
\end{equation}
Therefore,  when $d > 1$, the density $\delta_K(f(K),K)$ exists and equals zero.
\end{prop}


The following corollary yields Theorem \ref{T:density of m-sect}.


\begin{kor}~\label{C:m-sect estimate} Let $K$ be as in Proposition \ref{P:polynomial image estimate}, and suppose additionally that $K\subset \R$.  Assume that $m$ is a positive integer, and let $m_{odd}$ be its maximal odd divisor.  Then, there exist positive  real numbers $B_2$ and $E_2$, depending only on $m$ and $K$, such that if $B\geq B_2$, then
\[\delta_K(\mathbf{m-Sect}, [-1,1]; B) \leq E_2\cdot B^{(2/m_{odd} )-2}.\]
Therefore,  $\delta_K(\mathbf{m-Sect}, [-1,1])$ exists  for all $m$.  When $m_{odd}=1$ ---i.e., $m$ is a power of two--- we have already observed that the density equals 1.  When $m_{odd}> 1$, the above inequality implies that $\delta_K(\mathbf{m-Sect},[-1,1])=\lim_{B\rightarrow\infty}\delta_K(\mathbf{m-Sect}, [-1,1]; B)=0$.
\end{kor}

  The proofs of Proposition \ref{P:polynomial image estimate} and   Corollary \ref{C:m-sect estimate} are given in \S 3.


  \noindent\textbf{Remarks:}\quad (a)\quad  The estimates in the above proposition and corollary are based on a result of S. Lang (cf. \S 3.4, p.12) and a theorem of S. Schanuel  (cf. \S  3,  p.13, for a special case of this theorem).  The values of the  constants $B_0, E_0, B_2, E_2$ that appear above are not needed in this paper.  However, explicit values for $E_0$ and $E_2$ can be obtained using a more direct, detailed proof than Lang's.  The author plans a later paper in which these values appear. The values of $B_0$ and $B_2$ are more elusive, being absorbed in the ``big oh'' notation used in Schanuel's Theorem. If Theorem \ref{T:density of m-sect} and Corollary \ref{C:m-sect estimate} can be extended to the case in which the number field $K$ is replaced by  $\overline{\Q}$,  it will probably require a better understanding of $B_0$ and $B_2$, which will require a close analysis of the proof of Schanuel's Theorem.


(b)\quad The author wishes to thank Michael Stillman for a number of helpful conversations and for assisting with a series of  computations using the program \emph{Macaulay}.  These eventually suggested that something like Theorem 1.4 should be true.

\V

\tableofcontents

\section{M-sectability of angles and Chebyshev polynomials}

\subsection{M-sectable angles and constructibility}\quad We first establish some notation and conventions about angles and then  remind the reader of some basic facts about constructibility.


  We identify the Cartesian plane with the complex numbers and the $X$-axis with the real numbers in the usual way. \emph{From now on, when we wish to talk about an angle, we use instead the complex number on the unit circle, with which it is often identified.} Thus, we refer to a unit-length complex number $\alpha=e^{2\pi i\theta}$ as an \emph{angle}, rather than using $\theta$. This comports more smoothly with our algebraic arguments than the conventional terminology.
  Accordingly,  angle sums will be products of  unit complex numbers and angle multiples will be powers. Further, we usually  refer to the real and imaginary parts of $\alpha$ instead of to $\cos(\alpha)$ and  $\sin(\alpha)$ (or     $\cos(\theta)$ and $\sin(\theta)$).


  Given a set $S$ of complex numbers that includes the numbers $0$ and $1$, we say that a complex number $\beta$ is \emph{constructible over S} if there exists an unmarked straightedge and compass construction starting with the numbers in $S$ and ending with $\beta$. \emph{When $S=\{0,1\}$, we say simply that $\beta$ is constructible}.


  Let $RI(S)$ denote the set of real and imaginary parts of the numbers in $S$, regarded either as points on the X-axis, say, or as real numbers.  Clearly a complex number $\beta$ is constructible over $S$ if and only if it is constructible over $RI(S)$, and this is true if and only if both the real and imaginary parts of $\beta$ are constructible over $S$  (equivalently, over $RI(S)$, or equivalently, over the field $\Q(RI(S))$). When $\beta$ is an angle, the constructibility of either $Re(\beta)$ or $Im(\beta)$ implies the constructibility of the other, hence of $\beta$. Thus, for example, the angle $\beta$ is constructible over $\{0,1,\alpha\}$ if and only if $Re(\beta))$  is constructible over the field $\Q(Re(\alpha) )$.


  The \emph{Fundamental Theorem of Constructible Numbers} asserts that a real number $r$ is constructible over a subfield $F \subseteq \R$\ if and only if there is a finite tower of field extensions $F=F_0\subset F_1\subset\ldots\subset F_k$ such that: (i) $r\in F_k\subset \R$; (ii) each extension $F_i\subset F_{i+1}$ is quadratic.  In particular, this implies that numbers constructible over $F$ have minimal polynomials over $F$ with degrees that are  powers of two.


Given a positive integer $m$, we say that \emph{an angle $\alpha$ is \emph{$m$-sectable}} if there exists an angle $\beta$ satisfying $\beta^m=\alpha$  such that $\beta$ is constructible over  $\{0,1,\alpha\}$, or, equivalently, $Re(\beta)$ is constructible over $\Q(Re(\alpha))$.  Note that our conventions allow $\beta$ to be in quadrants other than the first even when $\alpha$ is an acute angle.


\subsection{Chebyshev polynomials}

A good reference for the standard definitions, properties, and examples of Chebyshev polynomials is \cite{w}.  Here we present some basic definitions and facts about these polynomials, tailored to our needs in this paper. We prove some results about them that may not be so well known.


Let $u$ and $v$ be indeterminates, and consider the ring $\C[u,v]$.  Then, there exist unique  polynomials $A_m(u,v)$ and $B_m(u,v)$ in $\Z[u,v]$ such that
\begin{equation}\label{E:formal power}
 (u+iv)^m=A_m(u,v^2)+ivB_m(u,v^2)
 \end{equation}
 in $\C[u,v]$.
Let $x$ and $y$ be any complex numbers, and substitute $x$ for $u$ and $y$ for $v$ in (\ref{E:formal power}):
\begin{equation}
(x+iy)^m= A_m(x,y^2)+iyB_m(x,y^2).
\end{equation}

The \emph{$m^{th}$ Chebyshev polynomials (of the first and second kind)} are now defined as follows.  For any real number $x$, choose a complex number $y$ so that $y^2=1-x^2$.  Then we set

\begin{equation}\label{E:Chebyshev definition}
T_m(x) = A_m(x,1-x^2)\quad\mbox{and}\quad U_m(x)= B_m(x, 1-x^2).\footnote{Our notation for $U_m$ is non-standard; it is usually denoted $U_{m-1}$.}
\end{equation}

Explicit formulas for $T_m(x)$  and $U_m(x)$ can be derived from the above:

\begin{eqnarray}
T_m(x) & = & \sum_{0\leq 2k\leq m}\sum_{l=0}^k(-1)^{k+l}{ m \choose 2k}{k \choose \ell}x^{m-2k+2\ell}.~\label{E:chebyshev one}\\
U_m(x) & = & \sum_{1\leq 2k+1\leq m}{m\choose 2k+1}x^{m-2k-1}(x^2-1)^k.~\label{E:chebyshev two}
\end{eqnarray}

Here are some examples of $T_m(x)$ for small values of $m$: $T_0(x)=1,\;\mbox{$T_1(x)=x$},\linebreak T_2(x)=2x^2-1,\;T_3(x)=4x^3-3x,\; T_4(x)=8x^4-8x^2+1,\; \mbox{$T_5(x)=16x^5-20x^3+5x$},\linebreak T_6(x)=32x^6-48x^4+18x^2-1, T_7(x)=64x^7-112x^5+56x^3-7x$.


\begin{lemma}\label{L:trig formula} \quad The following statements are equivalent for all angles $\alpha$ and $\beta$ and all positive integers $m$:  (a) $\beta^m =\alpha^{\pm1}$, and (b) $T_m(Re(\beta))=Re(\alpha)$.
\end{lemma}

The implication (a)$\Rightarrow$ (b) is, essentially,  the trigonometric identity that we mention in the introduction in the case $m=3$.  The proof of the equivalence is an easy derivation from the definitions, which we leave to the reader.


Next, here is a lemma listing several other useful properties of the polynomials $T_m(x)$.

\begin{lemma}\label{L:properties of T} \begin{enumerate}
\item The leading term of $T_m(x)$ is $2^{m-1}x^m$.
\item If $m$ is odd, $T_m(x)$ is an odd function of $x$, so that $T_m(0)=0$. Moreover, $T_m(\pm 1)=\pm 1$.  When $m$ is even, $T_m(x)$ is an even function of $x$.  Also then $T_m(0)=(-1)^{m/2}$ and $T_m(\pm 1)=1$.
\item For any positive integers $r$ and $s$,  $T_{rs}(x)=T_r(T_s(x))$.
\end{enumerate}
\end{lemma}

These statements  are well known; in any case, they can be verified easily from the definitions.

\begin{lemma}\label{L:chebyshev preserves unit interval} Let $m$ be any positive integer.  For any real number $x$,  $|x|\leq 1$ if and only if $|T_m(x)|\leq 1$.
\end{lemma}


\noindent\textbf{Remark}:\quad This lemma gives a technical fact about $T_m(x)$ that will be helpful when the sets we are estimating are contained in $[-1,1]$.

\begin{proof}\quad  The ``conjugate'' of identity (\ref{E:formal power}) is
\[(u-iv)^m = A_m(u,v^2)-ivB_m(u,v^2).\]
 Multiplying it by (\ref{E:formal power}) yields the identity
\[(u^2+v^2)^m=A_m(u,v^2)^2+v^2B_m(u,v^2)^2.\]
 Now let $I$ be the ideal in $\C[u,v]$ generated by $u^2+v^2-1$.  The quotient $\C[u,v]/I$ is a ring $R$ generated by the images $s$ of $u$ and $t$ of $v$, which satisfy $s^2+t^2=1$. The displayed identity above becomes
  \[1=A_m(s,1-s^2)^2+(1-s^2)B_m(s,1-s^2)^2\]
  in $R$. Choose any real number $x$ and then any complex number $y$ such that $x^2 +y^2=1$.  There is then a unique homomorphism of $R$ to \C\  sending $s$ to $x$ and $t$ to $y$, under which the last
identity above gets mapped to,
\[ 1 = T_m(x)^2+(1-x^2)U_m(x)^2.\]
 The polynomials $T_m$ and $U_m$ are real polynomials in the real variable $x$. It is immediate that  if $|x|\leq 1$, then $|T_m(x)|\leq 1$  and if $|T_m(x)|<1$ then $|x|< 1$. Finally, suppose that $|T_m(x)|=1$. Then, by the above equation, either $|x|=1$ or $U_m(x)= 0$. But, by inspecting equation (\ref{E:chebyshev two}), we see that $U_m(x)$ cannot be zero for $|x| > 1$. Thus, in any case, $|x|\leq 1$, as required.

\end{proof}


\begin{lemma}\label{L:properties of P}
\begin{enumerate}

\item If $m$ is an odd prime, then except for the leading coefficient and possibly the constant term $-a$, every coefficient in $T_m(x)-a$ is divisible by $m$.
\item If $m$ is  prime, then there exist infinitely many values of $a\in\Q$ such that $T_m(x)-a$ is irreducible over $\Q$.
\item If $m$ is  prime and $a$ is a transcendental number, then $T_m(x)-a$ is irreducible over $\Q(a)$.

\end{enumerate}
    \end{lemma}
\begin{proof}   Statement (a) may not be  widely known, but it is immediate from inspection of the coefficients in the formula for  $T_m(x)$.

We prove statement (b) in two parts.  First, when $m=2$, then $T_m(x)-a=2x^2-1-a$, which is clearly irreducible for all $a\in \Q$ such that $(1+a)/2$ is not a square in \Q. Secondly, suppose that $m$ is an odd prime.  Choose any rational value $a=r/s$ such that $r$ and $s$ are coprime and $r$ is divisible by $m$ but not by $m^2$.  Then, using statement (a), we may apply Eisenstein's Criterion together with the Gauss Lemma to conclude  that $T_m(x)-a$ is irreducible over $\Q$.

We now show that statement (c) follows from statement (b) via a somewhat standard argument. Let $t$ be an indeterminate, and consider  $T_m(x)-t$ as a polynomial in $\Q[t][x]$. Suppose that it factors in $\Q[t][x]$, say $T_m(x)-t=FG$, where both $F$ and $G$ are polynomials of positive degree in $x$, with coefficients $c_i$ and $d_j$, respectively,  that are polynomials in $\Q[t]$.   Statement (b) implies that we may choose a rational number $a$ such that $T_m(x)-a$ is irreducible over $\Q$ and such that $a$ is not a zero of any non-zero $c_i$ or $d_j$.  Now define a $\Q$-algebra homomorphism $\Q[t]\rightarrow \Q$ by sending $t$ to $ a$.  This induces a homomorphism
$\Q[t][x]\rightarrow \Q[x]$ which sends $T_m(x)-t$ to $T_m(x)-a$.  It also sends $F$ and $G$ to positive-degree polynomials in $\Q[x]$ whose product is $T_m(x)-a$, a contradiction. Therefore, $T_m(x)-t$ is irreducible over $\Q[t]$, hence over $\Q(t)$. Now suppose that $a$ is a transcendental number.  The rule $t\mapsto a$ defines a $\Q$-algebra isomorphism $\Q[t]\rightarrow \Q[a]$, hence an isomorphism $\Q(t)[x]\rightarrow \Q(a)[x]$. Obviously $T_m(x)-t\mapsto T_m(x)-a$, so the latter is irreducible over $\Q(a)$.
\end{proof}


The following lemma   proves Theorem \ref{T:algebraicity}.


\begin{lemma}\label{L:m-sectable angles}Let $\alpha$ be any angle, and set $a=Re(\alpha)$. \begin{enumerate}
 \item Assume that $m$ is not a power of two. If  $\alpha$ is $m$-sectable, then $T_m(x)- a$ is reducible over the field $\Q(a)$.
 \item Let $m$ be as in (a), and suppose that $\alpha$ is $m$-sectable. Then
 $a$ is an algebraic number.
 \item   If every angle is $m$-sectable, then $m$ is a power of two, and conversely.
 \end{enumerate}
     \end{lemma}

\begin{proof} (a)  By hypothesis, there is an angle $\beta$ that is constructible over the set $\{0,1,\alpha\}$ such that
$\beta^m =\alpha$. As we comment above, $b =Re(\beta)$ is constructible over the field $\Q(a)$.  Let $g(x)$ be the minimal polynomial of $b$ over $\Q(a)$.  From The Fundamental Theorem of Constructible Numbers, we know that the degree of $g(x)$ is a power of two.  Since $b$ is  a zero of $T_m(x)-a$, $g(x)$ must divide $T_m(x)-a$.  Moreover, it is a proper divisor because $m$ is not a power of two. This proves statement (a).

(b)  If $m$ is not a power of two, then it has an odd prime divisor, say $p$; write $m=kp$.  If $\alpha=\beta^m$, for some angle $\beta$ constructible over $\{0,1,\alpha\}$, then $\alpha=(\beta^k)^p$, with $\beta^k$ constructible over
$\{0,1,\alpha\}$, i.e., $\alpha$ is $p$-sectable. Further $p$ is not a power of two, so the hypotheses of statement (a) are satisfied for $p$ and $\alpha$. Hence $T_p(x)-a$ is reducible over $\Q(a)$.  Therefore, by statement (c) in Lemma \ref{L:properties of P}, $a$ must be algebraic.

(c) If every angle is $m$-sectable, there are $m$-sectable angles $\alpha$ for which $a$ is transcendental.  For this not to contradict statement (b), it must be the case that $m$ is a power of two.  Conversely, when $m$ is a power of two, every angle is $m$-sectable because bisection always holds.

\end{proof}

\subsection{Proof of Theorem \ref{T:constructible zeros} (a): }

We start with the case $m=2$.  The Chebyshev polynomial $T_2(x)$ equals $2x^2-1$.  Thus, for example, the equation $T_2(x)-1/4=0$ has solutions $\pm \sqrt{5/8}$, both of which are constructible over $\Q(1/4)=\Q$\ but do not belong to \Q.
Next, we apply the identity $T_m\circ T_n= T_{mn}$ to $T_{2^k}$, for $k>1$:
$T_{2^k}(x)- 1/4= 0$ if and only if $T_{2^{k-1}}(x) = \pm\sqrt{5/8}$, which has no rational solution.  However, it is easily checked inductively that solutions exist and are constructible.

In the rest of this subsection, therefore, \emph{we  assume that $m$ is an even number that is not a power of two}.

The remainder of the proof of Theorem \ref{T:constructible zeros} (a) makes use of the standard valuation $\nu_q$ on \Q, defined for every prime $q\in \N$.  Specifically,
$\nu_q(0)=0$ and $\nu_q(q^eh/k)=q^{-e}$, where $h,k,e$ are integers, and $h$ and $k$ are not divisible by $q$.


\begin{lemma}\label{L:calculating the valuation}\quad  Suppose that $r\in \Q$ and that $q$ is an odd prime.  (a) If $\nu_q(r)\leq 1$, then $\nu_q(T_m(r))\leq 1$.\quad (b) If $\nu_q(r) > 1$, then $\nu_q(T_m(r))=\nu_q(r)^m$.
\end{lemma}

\begin{proof}\quad  By Lemma \ref{L:properties of T} (a), we may write
\[ T_m(x) = 2^{m-1}x^m+c_1x^{m-1}+\ldots + c_{m-1}x+c_0,\]
for some $c_i\in \Z$.  Write $r=c/d$,  where $c$ and $d$ are relatively prime integers.  Then we have
\[ T_m(r)= \left(2^{m-1}c^m+dc_1c^{m-1}+\ldots +d^{m-1}c_{m-1}c+d^mc_0\right)/d^m,\]
and so
\[ \nu_q(T_m(r))= \nu_q(2^{m-1}c^m+dc_1c^{m-1}+\ldots +d^mc_0)\nu_q(d^{-m})\leq \nu_q(d)^{-m}.\]
When $\nu_q(r)\leq 1$, $q$ cannot divide $d$, and so  $\nu_q(d)=1$, proving (a).

Now suppose $\nu_q(r) > 1$, which implies that $q$ divides $d$ but does not divide $c$.  Then,
$q$ does not divide $2^{m-1}c^m$, and hence, it cannot divide $numerator(T_m(r))$.  So $numerator(T_m(r))\neq 0$, and
\[ \nu_q(T_m(r)) = 1\cdot\nu_q(d)^{-m} = \nu_q(r)^m,\]
which completes the proof of statement (b) of the lemma.
\end{proof}


 We now complete the proof of Theorem \ref{T:constructible zeros} (a).  The even number $m$ can be written as $2^kn$, for some integer $k\geq 1$ and some odd  integer $n$, which is $> 1$ because $m$ is not a power of two.  Then

 \[T_m(\sqrt{2/3})=T_{2^{k-1}n}(T_2(\sqrt{2/3}))=T_{2^{k-1}n}(1/3)\in \Q.\]

 Set this number equal to $a$.  So, $a\in \Q\cap [-1,1]$ (Lemma \ref{L:chebyshev preserves unit interval}), and $T_m(\sqrt{2/3})=a$, where $\sqrt{2/3}$ is constructible
 over $\Q(a)=\Q$.  It remains to show that no rational $r$ satisfies $T_m(r)=a$.

 First, since $\nu_3(1/3)=3 > 1$, we may apply Lemma \ref{L:calculating the valuation} (b) to $T_{2^{k-1}n}(1/3)=a$ : we get $\nu_3(a)=3^{2^{k-1}n}$.

Now suppose there is an  $r\in\Q$ such that $T_m(r) = a$. If $\nu_3(r)\leq 1$,  Lemma \ref{L:calculating the valuation} (a) implies that $\nu_3(T_m(r))\leq 1$, which contradicts
 $\nu_3(a)=3^{2^{k-1}n}$.  Therefore, we must have  $\nu_3(r)>1$. We then apply Lemma \ref{L:calculating the valuation} (b) again and get
 \[ \nu_3(r)^{2^kn}=\nu_3(T_m(r))=\nu_3(a)=3^{2^{k-1}n}.\]
 Taking $(2^{k-1}n)^{th}$ roots of both sides of this equality, we get $\nu_3(r)^2=3$, which is impossible.  Therefore, no $r\in \Q$ can satisfy $T_m(r) = a$, completing the proof of Theorem \ref{T:constructible zeros} (a). \quad\quad $\square$


 \subsection{Toward a proof of Theorem \ref{T:constructible zeros} (b)}\quad The key idea in the proof is to transform the assertion, which involves real solutions to certain real polynomial equations, to an equivalent one about complex solutions to certain complex polynomial equations. That  assertion becomes relatively easy to verify. We  develop the technical tools that allow us to make and use this transformation
 via a number of lemmas, the key one being Lemma \ref{L:tower lifting}.


 \begin{lemma}\label{L:degree-two extension} Let $\alpha$ be any angle, and let $a=Re(\alpha)$. For any subfield $F\subseteq\C$, $F(\alpha)$ is an extension of $F(a)$, and $[F(\alpha):F(a)]\leq 2$.  When $F\subseteq\R, \quad  [F(\alpha):F(a)]=2 \Leftrightarrow \alpha\neq \pm 1$.   Finally,  when $F\subseteq \R$, $F(a)=F(\alpha) \cap \R$.
 \end{lemma}
 \begin{proof}   We note that $a=(\alpha + \alpha^{-1})/2$, showing that $a \in F(\alpha)$ and that $\alpha$ is a zero of $x^2-2ax+1$. This proves the first two assertions. The degree $[F(\alpha):F(a)]<2$ if and only if the two fields are equal. When $F\subseteq\R$, this can happen if and only if the unit complex number $\alpha$ is real, i.e., $\alpha=\pm 1$. Still assuming $F\subseteq\R$, we have  $F(a)\subseteq F(\alpha) \cap \R \subseteq F(\alpha)$.  If $\alpha = \pm 1$, these are all equalities.  Otherwise, $F(\alpha)$ is a degree-two extension of each of the other fields. \end{proof}


 \begin{lemma}\label{L:comparing definite solutions}  Suppose that the angles $\alpha$ and $\beta$ satisfy $\alpha\neq\pm 1$ and $\beta\neq\pm 1$, and let $a=Re(\alpha)$  and $b=Re(\beta)$. Assume further that $\beta^m=\alpha^{\pm 1}$ (equivalently $T_m(b)=a$).  Then, $\beta \in \Q(\alpha) \Leftrightarrow b\in \Q(a)$.
 \end{lemma}
 \begin{proof}\quad The hypotheses imply that we have a commutative diagram
  \[ \begin{CD}
    \Q(\alpha) @>>> \Q(\beta)\\
    @AAA       @AAA \\
    \Q(a) @>>> \Q(b)
     \end{CD} \]
 with arrows representing field extensions; the vertical arrows represent extensions of degree two.  We then have:
 \[ \beta\in \Q(\alpha)\Leftrightarrow \Q(\beta)=\Q(\alpha)\Leftrightarrow [\Q(\beta):\Q(a)]=2\Leftrightarrow \Q(b)=\Q(a)\Leftrightarrow b\in\Q(a).\]

 \end{proof}


 \noindent\textbf{Remark:}\quad When $\alpha=\pm 1$, Lemma \ref{L:comparing definite solutions} is false.
 For example, say $\alpha=1$ and $m=3$.  Take $\beta$ to be a non-real cube root of unity.  Then $b=-1/2\in \Q(a)=\Q$, but $\beta\not\in\Q(\alpha)=\Q$.


\begin{corlem}\label{C:comparing indefinite solutions}\quad  Suppose that $\alpha$ is an angle $\neq \pm 1$  and $a=Re(\alpha)$.  Then some zero of $z^m-\alpha$ belongs to $\Q(\alpha)$ if and only if some zero of $T_m(x)-a$ belongs to $\Q(a)$.
\end{corlem}

\begin{proof}

\noindent $\Rightarrow:$\quad  Suppose $\beta^m=\alpha$, and $\beta\in \Q(\alpha)$.  Then,  Lemma \ref{L:trig formula} gives $T_m(Re(\beta))=a$. Since $\beta\neq\pm1$, Lemma \ref{L:comparing definite solutions} applies to give $Re(\beta)\in \Q(a)$.

\noindent $\Leftarrow:$ \quad Suppose $T_m(b)=a$, for $b\in \Q(a)$.  Since $|\alpha|=1$, we have $|a|\leq 1$.  Therefore, by Lemma \ref{L:chebyshev preserves unit interval},  $|b|\leq 1$, so that there is an angle $\beta$ such that $Re(\beta)=b.$
 Lemma \ref{L:trig formula} gives $\beta^m=\alpha^{\pm 1}$. Since $\alpha\neq\pm 1$, then $\beta\neq\pm 1$. Therefore,  Lemma \ref{L:comparing definite solutions} applies to give $\beta\in\Q(\alpha)$.  \end{proof}


The next  lemma gives our main technical construction.

\begin{lemma}[\textbf{tower-lifting}]\label{L:tower lifting}\quad Suppose $\alpha$ and $\beta$ are angles, with $\alpha,\beta \neq \pm 1$,  and let $a$ and $b$ their real parts, respectively.  Suppose that $a= T_m(b)$ (equivalently, $\beta^m=\alpha^{\pm 1}$ ), so that $\Q(a)\subseteq\Q(b)$ and $\Q(\alpha)\subseteq\Q(\beta)$. Finally, assume that $b$ is constructible over $\Q(a)$. Then there exists a commutative diagram of towers of degree-two extensions
\[ \begin{array}{ccccccccccc}
\Q(\alpha)& = &G_0 & \subset & G_1 & \subset & \ldots & \subset & G_k & = & G_{k-1}(\beta)\\
        & &\uparrow & &  \uparrow & & \ldots & &\uparrow & & \\
\Q(a)& = & F_0 & \subset & F_1 & \subset & \ldots & \subset & F_k & = & F_{k-1}(b)
    \end{array}, \]
in which  each vertical arrow is an inclusion map of a degree-two field extension.
\end{lemma}

\begin{proof} \quad Since $b$ is constructible over $\Q(a)$, the Fundamental Theorem of Constructible Numbers asserts that there exists a tower of degree-two field extensions
\[ \Q(a)=F_0\subset F_1\subset \ldots\subset F_k,\]
with $F_k\subset \R$ and $b\in F_k$. We choose such a tower with $k$ minimal.  Then, $b \notin F_{k-1}$, so $F_k= F_{k-1}(b)$, as required.

Define  $G_i$ by $G_i=F_i(\alpha)$. This defines inclusions  designated by the arrows in the diagram.  Since $F_{i-1}\subset F_i$, for each $i= 1,\ldots, k$, we have $G_{i-1}\subseteq G_i$, hence a commutative diagram of
inclusions
(or extensions), as pictured.  Lemma \ref{L:degree-two extension} immediately gives $[G_i:F_i]=2$, for all $i$.

We now compute
\[2[G_i:G_{i-1}] = [G_i:G_{i-1}][G_{i-1}:F_{i-1}]= [G_i:F_{i-1}] = [G_i:F_i][F_i:F_{i-1}]= 4.\]
Therefore, each $G_i$ is a degree-two extension of $G_{i-1}$, as claimed.

Next, we have
\[ F_{k-1}(b)=F_k\subset G_k=F_k(\alpha)=F_{k-1}(b,\alpha)\subseteq F_{k-1}(\beta,\alpha)= F_{k-1}(\beta),\]
from which we extract $F_{k-1}(b)\subset G_k\subseteq F_{k-1}(\beta)$.
Since $[F_{k-1}(\beta):F_{k-1}(b)]=2$ (Lemma \ref{L:degree-two extension})
and $F_{k-1}(b)\neq G_k$, we conclude that $G_k=F_{k-1}(\beta)$, hence, that $\beta \in G_k$.

Finally, we argue that $\beta\notin G_{k-1}$.  For  if $\beta\in G_{k-1}$, then $b= Re(\beta)\in G_{k-1}\cap\R = F_{k-1}$ (Lemma \ref{L:degree-two extension}),  contradicting the minimality of $k$. Therefore, $G_k =G_{k-1}(\beta)$, completing the proof of the lemma.
 \end{proof}


\subsection{Proof of Theorem \ref{T:constructible zeros}(b)}  \begin{proof}

  Recall that Theorem \ref{T:constructible zeros}(b) asserts that, for $a\in [-1,1]$ and $m$ odd, the equation $T_m(x)= a$ has a solution constructible over $\Q(a)$\quad $\Leftrightarrow$\quad it has a solution in $\Q(a)$. Since the implication $\Leftarrow$ is trivial, we need only prove the implication $\Rightarrow$.

  We first dispose of a simple special case.  Suppose $a=\pm 1$.  Then $T_m(x)=\pm 1$ has a solution in $\Q(\pm1)=\Q$, namely $x=\pm 1$ (Lemma \ref{L:properties of T} (b)). So the desired implication is trivially true.  For the rest of the proof we assume that $a\neq \pm 1$.

 Let $\alpha$ be an angle such  that $a= Re(\alpha)$. Of course, then  $\alpha\neq \pm 1$. If the real number $b$ satisfies  $T_m(b)=a$, then $b\in (-1,1)$ (Lemma \ref{L:chebyshev preserves unit interval}), and $\Q(a)\subseteq \Q(b)$.   There is then an angle $\beta\neq \pm 1$ such that $b= Re(\beta)$ and  $\beta$ is a zero of the polynomial $z^m-\alpha$ (Lemma \ref{L:trig formula}). Therefore, $\Q(\alpha)\subseteq\Q(\beta)$.

  We now use the hypothesis that some solution of $T_m(x) = a$ is constructible over $\Q(a)$, letting $b$ be that solution and $\beta$ as described above. We may then apply tower-lifting (Lemma \ref{L:tower lifting}).  Since $b\in F_k$ and $\beta\in G_k$, we have a commutative diagram of field extensions
  \[ \begin{array}{ccccccc}
 G_0 & = & \Q(\alpha) & \rightarrow & \Q(\beta) & \rightarrow &  G_k\\
        & &\uparrow & &  \uparrow & & \uparrow  \\
 F_0 & = & \Q(a) & \rightarrow & \Q(b) & \rightarrow &   F_k
    \end{array}, \]
 in which each vertical arrow represents a degree-two extension.  The diagram implies that $[\Q(b):\Q(a)]=[\Q(\beta):\Q(\alpha)]$ and that this quantity divides $[G_k:G_0]=[F_k:F_0]=2^k$.  Therefore, $[\Q(b):\Q(a)]=[\Q(\beta):\Q(\alpha)]= 2^j$, for some natural number $j\leq k$.

 Now let $f(z)$ be the minimal polynomial of $\beta$ over $\Q(\alpha)$.  Then, $f(z)$ has degree $2^j$, and $f(z)$ divides $z^m-\alpha$.

The next part of the proof, which essentially occurs in the field $\C$, is inspired by an argument of Van der Waerden (\cite{wae}, p.171).

We begin with a convenient listing  of the zeros of $z^m-\alpha$ : $\beta, \beta\xi,\ldots,\beta\xi^{m-1}$, where  $\xi$ is an arbitrary but fixed  primitive $m^{th}$ root of unity.  Therefore, since the zeros of $f(z)$ form a subset of the set of zeros of $z^m-\alpha$,   we may write the former as $\beta, \beta\xi_1,\ldots, \beta\xi_{2^j-1}$, where the $\xi_i$'s are distinct $m^{th}$ roots of unity ($\neq 1$).  Let $\lambda$ be the constant term of $f(z)$.  Then $\lambda\in \Q(\alpha)$ and $\lambda= \beta^{2^j}\cdot\xi'$, where $\xi'$ is some $m^{th}$ root of unity.  Therefore, $\lambda^m=\beta^{m2^j}=\alpha^{2^j}$.

Next, recall that $m$ is odd, and so $2^j$ and $m$ are relatively prime. Therefore, we have integers $r$ and $s$ such that $2^jr+ms=1$.  It follows that $\alpha=\alpha^{2^jr}\alpha^{ms}=\lambda^{mr}\alpha^{ms}=(\lambda^r\alpha^s)^m$.
Set $\gamma=\lambda^r\alpha^s\in \Q(\alpha)$. Then,  $\gamma^m=\alpha$, i.e., some solution of $z^m-\alpha$ is in $\Q(\alpha)$.

  Finally, we use Corollary \ref{C:comparing indefinite solutions} to return to the original context in $\R$, concluding that some  solution of $T_m(x)-a$ is in $\Q(a)$. This completes the proof.\end{proof}


\subsection{Proof of Corollary \ref{C:generalized wantzel}(a) and the Generalized Wantzel Theorem}\quad \begin{proof}\quad
(a) Let $a\in [-1,1]\cap\Q$ be as given in Theorem \ref{T:constructible zeros} (a), so that $T_m(x)-a$ has a zero constructible over $\Q(a)=\Q$ but no zero in \Q. Note that the selected $a$ satisfies $-1\leq a\leq 1$. Let $\alpha$ be an angle such that $Re(\alpha)= a$.  Then $\alpha$ is $m$-sectable, by Proposition \ref{P:constructibility criterion}.

(b)  Combining Proposition \ref{P:constructibility criterion} with Theorem \ref{T:constructible zeros} (b), we see that $\alpha$ is\\ $m_{odd}$-sectable if and only if $T_{m_{odd}}(x)-a$ has a zero in $\Q(a)$.  The desired result now follows from the observation that $\alpha$ is $m$-sectable if and only if it is $m_{odd}$-sectable.  (In one direction this is obvious; in the other direction it is true because bisection always holds.) \end{proof}

\V

\section{Density}

\subsection{Absolute values}\quad  As we mention in \S 1, the notion of density that we use is based on the number-theoretic concept of \emph{height} of an algebraic number, which, in turn, may be defined in terms of  a standard set of so-called \emph{absolute values} on a number field.  For example, on the field of rational numbers \Q\, we have the usual absolute value $|\;\;|$, and we have an absolute value $|\;\; |_p$ defined as follows, for each prime number $p$ in \Q: $|0|_p=0$, and $|p^ec/d|=p^{-e}$, for all integers $c,d,e$, with $c$ and $d$ not divisible by $p$.  This collection of absolute values, which we denote by $M_{\Q}$, is often called the canonical set of absolute values on $\Q$. (Of course the absolute value $|\;\;|_p$ is just the valuation $\nu_p$ used earlier but under an alternative name.)

Every number field $K$ has a similar canonical set of absolute values $M_K$, which can be defined as the set of all  extensions of the absolute values in $M_\Q$ to $K$.  Details concerning how this is done can be found, say, in \cite{lan}, Ch.2.  Here we list only some key properties of the absolute values in $M_K$, which are usually denoted by $|\;\; |_v$ or simply by $v$:  In the following, $x$ and $y$ are assumed to range over the field $K$.

\begin{enumerate}
\item $|x|_v\geq 0$, with equality if and only if $x=0$.
\item $|xy|_v=|x|_v|y|_v$.
\item $|x+y|_v\leq |x|_v +|y|_v$.\\
The foregoing define the general concept of an absolute value.
\item For each $|\;\; |_v$ there is a natural number $n_v$, called its weight, such that, for each non-zero $x\in K$, $\prod_{v\in M_K}|x|_v^{n_v} =1$.
\end{enumerate}

It is an easy exercise to see that the canonical set $M_{\Q}$ satisfies the above properties, where the weights all equal $1$.

\subsection{Height} \quad Given a point $P=[x_0,\dots,x_n]$ in the projective space $P^n(K)$, one defines its height as follows:
\[ H_K(P)= \prod_{v\in M_K}\sup\{|x_0|_v^{n_v},\ldots,|x_n|_v^{n_v}\}.\]

Properties (b) and (d) above imply that $H_K(P)$ is well-defined.

Since we are interested only in the special case of $n=1$, in fact in the heights of field elements of $K$, we rewrite the definition to focus on this.  We identify $K$ with the set of $[x_0,x_1]\in P^1(K)$ satisfying $x_0\neq 0$.  Then, we have, for each $x\in K$,
\[H_K(x)=\prod_{v\in M_K}\sup\{1,|x|_v^{n_v}\}.\]

The following properties of $H_K$ are important for us.
\begin{enumerate}
\item For all $x\in K$ and natural numbers $n,\; H_K(x^n)=H_K(x)^n.$
\item For all non-zero $x\in K$, $H_K(x^{-1})=H_K(x)$.
\item For every real $B\in[1,\infty)$, $H_K^{-1}([1,B])$ is finite.
\end{enumerate}

\subsection{Density}\quad  Given two sets $S\subseteq T$ of complex numbers , we now define the $K$-density of $S$ in $T$, as in the introduction.  To insure that the cardinality of $|T\cap H_K^{-1}([1,B])|$ is non-zero, we always assume that $1\in T$. Then, the  $K$-density of $S$ in $T$ is defined  to be the limit as $B\rightarrow \infty$ of the quotients of finite cardinalities,
\[ \delta_K(S,T;B)=\frac{|S\cap H_K^{-1}([1,B])|}{|T\cap H_K^{-1}([1,B])|},\]
provided this limit exists. We denote the limit by $\delta_K(S,T)$.

\subsection{Density of polynomial images}\quad  Now let $f(X)\in K[X]$ be a  degree $d\geq 1$ polynomial.
\begin{prop} [\cite{lan}, p. 82]\label{P:counting estimate} There exist positive real numbers $C_1$ and $C_2$, depending only on $f$ such that, for every $x\in K$,
\[ C_2^{-n}H_K(x)^d\leq H_K(f(x))\leq C_1^nH_K(x)^d,\]
where $n=[K:\Q]$.
\end{prop}


\begin{propcor}\label{C:bounding polynomial image finite density}  Given an algebraic number field $K$, a polynomial $f(X)\in K(X)$
of degree $d\geq 1$, and a real number $B\in [1,\infty)$, there exists a positive real number $C$, depending only on $f$ and $K$, such that
\[\delta_K(f(K),K;B)\leq \frac{|H_K^{-1}[1,CB^{1/d}]|}{|H_K^{-1}[1,B]|}.\]
\end{propcor}

\begin{proof}\quad It suffices to show that, for some $C$ as described above,
\[ |f(K)\cap H_K^{-1}[1,B]|\leq |H_K^{-1}[1,CB^{1/d}]|.\]
So, suppose $x\in K$, and $y=f(x)\in H_K^{-1}[1,B]$, i.e., $H_K(y)= H_K(f(x))\leq B$. Then, by Proposition \ref{P:counting estimate}, $ C_2^{-n}H_K(x)^d\leq B$, so that $x\in H_K^{-1}[1, (C_2^nB)^{1/d}]$. Applying $f$ to this membership relation, we get, $y=f(x)\in f(H_K^{-1}[1, (C_2^nB)^{1/d}])$.  Therefore, $f(K)\cap H_K^{-1}[1,B]\subseteq f(H_K^{-1}[1, (C_2^nB)^{1/d}])$, and, hence $|f(K)\cap H_K^{-1}[1,B]|\leq |f(H_K^{-1}[1, (C_2^nB)^{1/d}])|$. It remains only to set $C=(C_2)^{n/d}$ and to observe that $|f(H_K^{-1}[1, (C_2^nB)^{1/d}])|\leq |H_K^{-1}[1, (C_2^nB)^{1/d}]|= |H_K^{-1}[1,CB^{1/d}]|$.
\end{proof}


We now invoke and apply a (special case of a) theorem of S. Schanuel \cite{sch}.


\begin{theorem}[\textbf{Schanuel's Theorem}]\label{T:Schanuel's Theorem}  Let $K$ be an algebraic number field, and let $B$ be a real number $\geq 1$. Set $[K:\Q]=n$. Then, there exists a constant $\mathbf{S}_K$, depending only on $K$, such that
\[ |H_K^{-1}[1,B]|= \mathbf{S}_K\cdot B^2 + \mathbf{O}(C(n,B)),\]
where

\[ C(n, B) = \left\{\begin{array}
                    {r@{\quad:\quad}l}
                    B\log B & n=1\\B^{1/n} & n\geq 2
                    \end{array}
                    \right..   \]
\end{theorem}


\noindent\textbf{Remarks:} a)  Schanuel computes $\mathbf{S}_K$ explicitly in terms of standard numerical invariants of the field $K$ (see \cite{lan} or \cite{sch}).  For example, when $n=1$, $\mathbf{S}_K$ equals $6/\pi^2$.

b) The term ``$\mathbf{O}(C(n,B))$'' follows the standard ``big oh'' convention.


\begin{propcor}[\textbf{Proposition \ref{P:polynomial image estimate}}]\label{C:polynomial image finite density estimate}\quad  Let $K, f$, and $B$ be as in Corollary \ref{C:bounding polynomial image finite density} above.  Then, there exist positive real numbers $B_0$ and $E_0$, depending only on $K$ and $f$, such that, for $B\geq B_0$,
\[ \delta_K(f(K),K;B)\leq E_0\cdot B^{(2/d)-2}.\]
Therefore, when $d > 1$,  $\delta_K(f(K),K)=\lim_{B\rightarrow\infty}\delta_K(f(K),K;B)$ exists and equals zero.
\end{propcor}

We omit the proof, which is a simple computation using Corollary \ref{C:bounding polynomial image finite density} and Schanuel's Theorem.


\subsection{Intersecting with $[-1,1]$}\quad Our main application of the foregoing results involves angle cosines, i.e., real numbers lying in the interval $[-1,1]$.  Therefore, we should be estimating the size of the sets $H_K^{-1}[1, B]\cap [-1,1]$, as well as densities relative to these.  This subsection shows how to obtain these estimates quite easily  in terms of those for the sets $H_K^{-1}[1, B]$. \emph{Since we are dealing with real numbers in $H_K^{-1}[1, B]$, we shall assume in this subsection that $K$ is a  subfield of $\R$.} Let $K^{\ast}= K\setminus \{0\}$.

Let $I:K^{\ast}\rightarrow K^{\ast}$ denote the inversion $\alpha\mapsto \alpha^{-1}$.  It gives a bijection
\[ K^{\ast}\cap [-1,1]\rightarrow K^{\ast}\setminus (-1, 1),\]
where here $(-1,1)$ denotes the interior of the interval $[-1,1]$.  As noted above in \S 3.2, property (b), $H_K(\alpha)=H_K(I(\alpha)$), for all non-zero
$\alpha\in K$, so $I$ induces a bijection of finite sets
$ H_K^{-1}[1, B]\cap K^{\ast}\cap [-1,1]\rightarrow H_K^{-1}[1, B]\cap(K^{\ast}\setminus (-1,1))$.
These sets intersect in $\{-1,1\}$, and their union is $H_K^{-1}[1, B]\cap K^{\ast}=H_K^{-1}[1, B]\setminus\{0\}$. It follows that
\begin{equation}\label{E:intersecting with [-1,1]}
2|H_K^{-1}[1, B]\cap K^{\ast}\cap [-1,1]| - 2 = |H_K^{-1}[1, B]| -1.
\end{equation}
But $H_K^{-1}[1, B]\cap [-1,1] = (H_K^{-1}[1, B]\cap K^{\ast}\cap [-1,1])\cup \{0\}$, so $|H_K^{-1}[1, B]\cap [-1,1]| - 1 = |H_K^{-1}[1, B]\cap K^{\ast}\cap [-1,1]|$.  Combining this with equation (\ref{E:intersecting with [-1,1]}), we get  the following:
\begin{prop}\label{P:intersecting with [-1,1]} \quad For any real $B\geq 1$ and any number field $K\subset \R$,  \[|H_K^{-1}[1, B]\cap [-1,1]| = \frac{1}{2}(|H_K^{-1}[1, B]|+3).\quad  \square\]
\end{prop}


The following corollary is essentially Proposition \ref{P:polynomial image estimate} ``relativized down to $[-1,1]$.''


\begin{propcor}\label{C:intersecting with [-1,1]}\quad  Let $K, f$, $d$, and $B$ be as in Corollary \ref{C:bounding polynomial image finite density} above, with $K\subset \R$.   Assume additionally that, for all real $x$, $|x|\leq 1$ if and only if $|f(x)|\leq 1$ (as is the case for Chebyshev polynomials). Then, there exist positive real numbers $B_1$ and $E_1$, depending only on $K$ and $f$, such that, for $B\geq B_1$,
\[ \delta_K(f(K)\cap [-1,1],  [-1,1];B)\leq E_1\cdot B^{(2/d)-2}.\]
Therefore, when $d > 1$,  $\delta_K(f(K)\cap [-1,1],  [-1,1])$ exists and equals zero.
\end{propcor}

\begin{proof}\quad  Suppose, as in the proof of Corollary \ref{C:bounding polynomial image finite density}, $x\in K$ and $y = f(x)\in H_K^{-1}[1,B]$, but suppose additionally that $y\in [-1,1]$.  Then we can conclude that $x\in H_K^{-1}[1,(C_2^nB)^{1/d}]\cap [-1,1]$, and so $y=f(x)\in f(H_K^{-1}[1,(C_2^nB)^{1/d}]\cap [-1,1])$.  Therefore,
\[ |H_K^{-1}[1,B]\cap f(K)\cap [-1,1]|\leq |f(H_K^{-1}[1,(C_2^nB)^{1/d}]\cap [-1,1])|\leq |H_K^{-1}[1,(C_2^nB)^{1/d}]\cap [-1,1]|.\]
We use this inequality, as in Corollary \ref{C:bounding polynomial image finite density}, to get an upper bound for $\delta_K$:
\[\delta_K(f(K)\cap [-1,1], [-1,1])\leq \frac{|H_K^{-1}[1,(C_2^nB)^{1/d}]\cap [-1,1]|}{|H_K^{-1}[1,B]\cap[-1,1]|}.\]
Next , we substitute the  result of Proposition \ref{P:intersecting with [-1,1]} and simplify slightly:
\[\delta_K(f(K)\cap [-1,1], [-1,1])\leq \frac{|H_K^{-1}[1,(C_2^nB)^{1/d}]|+3}{|H_K^{-1}[1,B]|+ 3}.\]
At this point, we apply  Schanuel's Theorem to the numerator and denominator and conclude with a straightforward computation.
\end{proof}


\subsection{The K-density of m-Sect  in $[-1,1]$: proof of Corollary \ref{C:m-sect estimate}}\quad We begin by reminding the reader of how \textbf{m-Sect} $\cap\  K$ relates to the image of $K$ under the Chebyshev polynomial $T_{m_{odd}}$, where $m_{odd}$ is the largest odd divisor of $m$. Recall that $\mathbf{m-Sect}= \mathbf{m_{odd}-Sect}$.  Now choose any element $a \in \mathbf{m-Sect} \cap\ K =  \mathbf{m_{odd}-Sect} \cap\ K$. Then, the Generalized Wantzel Theorem implies that $a$ is in   $T_{m_{odd}}(K)\cap[-1,1]$.  That is,
\[ \textbf{m-Sect}\cap K\subseteq T_{m_{odd}}(K)\cap[-1,1].\]
It follows that
\[\delta_K(\textbf{m-Sect}, [-1,1]; B)\leq \delta_K(T_{m_{odd}}(K)\cap[-1,1],  [-1,1]; B).\]
Since $|T_{m_{odd}}(x)|\leq 1 \Leftrightarrow
 |x|\leq 1$ (Lemma \ref{L:chebyshev preserves unit interval}), we can use Corollary \ref{C:intersecting with [-1,1]} to  get an upper bound for the density $\delta_K(\textbf{m-Sect}, [-1,1]; B)$:  namely, there exist constants $B_2$ and $E_2$, depending only on $m$ and $K$, such that, for $B\geq B_2$,

\[ \delta_K(\textbf{m-Sect}, [-1,1]; B)\leq E_2\cdot B^{(2/m_{odd})-2}.\]

 This concludes our proof of Corollary \ref{C:m-sect estimate}.

 \V

\end{document}